\documentclass{amsart} 
\usepackage{epsf} 
\usepackage{amssymb,latexsym, amsmath, amscd, array, graphicx}

\swapnumbers 
\numberwithin{equation}{section} 

\theoremstyle{plain} 
\newtheorem{thm}{Theorem}[section] 
\newtheorem{cor}[thm]{Corollary} 
\newtheorem{theorem}[thm]{Theorem} 
\newtheorem{lem}[thm]{Lemma}

\newtheorem{fact}[thm]{Fact}

\theoremstyle{definition} 
\newtheorem{defin}[thm]{Definition} 
\newtheorem{usefulfacts}[thm]{Useful Facts}

\newtheorem{rem}[thm]{Remark} 
\newtheorem{example}[thm]{Example}

\newtheorem{remark}[thm]{Remark}

\def\Im{\protect\operatorname{Im}} 
 
\def\diam{\protect\operatorname{diam}}

\def\pr{\protect\operatorname{pr}}

\def\pr{\protect\operatorname{pr}}

\def\exp{\protect\operatorname{exp}}

\def\Q{{\mathbb Q}} 
\def\R{{\mathbb R}} 
\def\N{{\mathbb N}}

\def\1{\hbox{\rm\rlap {1}\hskip.03in{\rom I}}} 
\def\Bbbone{{\rm1\mathchoice{\kern-0.25em}{\kern-0.25em} 
{\kern-0.2em}{\kern-0.2em}I}} 

\def\wt{\widetilde}

\begin{document}
\title[Manifolds admitting a $Y^x$-Riemannian metric] 
{Topological Properties of Manifolds admitting a $Y^x$-Riemannian metric} 
\author[V.~Chernov, P.~Kinlaw, R.~Sadykov]{Vladimir Chernov, Paul Kinlaw, Rustam Sadykov}
\address{V.~Chernov, Mathematics Department, 6188 Kemeny,  Dartmouth College, Hanover NH 03755, 
USA} 
\email{Vladimir.Chernov@dartmouth.edu} 
\address{P.~Kinlaw, Mathematics Department, 6188 Kemeny,  Dartmouth College, Hanover NH 03755, 
USA} 
\email{Paul.Kinlaw@dartmouth.edu} 
\address{R.~Sadykov, Department of Mathematics, University of Toronto, Bahen Centre, 40 St.~George St., 
Toronto, Ontario M5S 2E4,
CANADA } 
\email{sadykov@math.toronto.edu}


\begin{abstract} A complete Riemannian manifold $(M, g)$ is a $Y^x_l$-manifold
if every unit speed geodesic $\gamma(t)$ originating at $\gamma(0)=x\in M$ 
satisfies $\gamma(l)=x$ for $0\neq l\in \R$. B\'erard-Bergery proved that 
if $(M^m,g), m>1$ is a $Y^x_l$-manifold, then $M$ is a closed manifold with finite fundamental group, and the cohomology ring $H^*(M, \Q)$ is generated by one element. 

We say that $(M,g)$ is a $Y^x$-manifold if for every $\epsilon >0$ there exists $l>\epsilon$ such that for every unit speed
geodesic $\gamma(t)$ originating at $x$, 
the point $\gamma(l)$ is $\epsilon$-close to $x$. 
We use Low's notion of refocussing Lorentzian space-times to show that 
if $(M^m, g), m>1$ is a $Y^x$-manifold, then $M$ is a closed manifold 
with finite fundamental group. As a corollary we get that a Riemannian covering of a $Y^x$-manifold is a 
$Y^x$-manifold. Another corollary is that if $(M^m,g), m=2,3$ is a 
$Y^x$-manifold, then $(M, h)$ is a $Y^x_l$-manifold for some metric $h.$ 
\end{abstract}


\maketitle

\section{Introduction}
Throughout the paper all Riemannian manifolds are assumed to be geodesically complete, 
while Lorentzian manifolds (see Section~\ref{LorentzGeometry}) are not assumed to be complete unless this is explicitly stated.
We will tacitly assume that a manifold $M$ under consideration is a smooth connected manifold without boundary (not necessarily compact or oriented). 

\begin{defin}[$Y^x_l$-manifolds]
Let $(M,g)$ be a Riemannian manifold, $x$ a point in $M$ and $l$ a nonzero real number. 
We say that $(M,g)$ is 
a {\it $Y^x_l$-manifold\/} if 
for every geodesic $\gamma:\R\to M$ satisfying $\gamma(0)=x$ and $|\dot \gamma (0)|=1$ we have 
$\gamma(l)=x.$ 
\end{defin}

In other words $(M, g)$ is a $Y^x_l$-manifold if each geodesic parametrized by arc length and emitted from 
$x$ comes back to $x$ at the moment $l$. Such manifolds attracted a lot of attention~\cite{Besse}. They are 
related to Blaschke manifolds and manifolds all of whose geodesics are closed.  The following 
weak form of the Bott-Samelson Theorem~\cite{Bott, Samelson} 
was proved by B\'erard-Bergery, see~\cite{BerardBergery}, 
~\cite[Theorem 7.37, page 192]{Besse}.

\begin{theorem}[B\'erard-Bergery]\label{BB1}
If $(M,g)$ is a $Y^x_l$-Riemannian manifold of dimension at least $2$, then $M$ is a closed manifold with 
finite fundamental group and the cohomology ring 
$H^*(M,\Q)$ is generated by one element.
\end{theorem}

The standard metric on $S^1$ shows that the statement of Theorem~\ref{BB1} 
is false for one dimensional $(M, g).$

\begin{remark}\label{questions}
Besse~\cite[Definitions 7.7]{Besse} describes a few notions closely related to $Y^x_l$-manifolds.
In particular, $(M, g)$ is a {\it $Z^x$-manifold\/} if all the geodesics starting at $x$ come back to $x$. 
Clearly every $Y^x_l$-manifold is a  $Z^x$-manifold. However according to~\cite[Question 7.70]{Besse} it is not 
known if every $Z^x$-manifold is a $Y^x_l$-manifold for some nonzero
$l.$ Moreover it is not known if the length of the first 
return to $x$ of a unit speed geodesic starting at $x$ 
is uniformly bounded for every initial unit vector in $T_xM$,~see~\cite[Question 7.71]{Besse}. 
\end{remark}

In this work we introduce and study the class of $\widetilde Y^x$-manifolds 
that generalizes $Y^x_l$-manifolds.
Remark~\ref{questions} suggests that a priori even quite simple questions about manifolds satisfying 
conditions close to the definition of the $Y^x_l$-manifold 
can be hard to answer. We show however that all $\widetilde Y^x$-manifolds satisfy a counterpart of the 
B\'erard-Bergery Theorem (see Theorem~\ref{maintheorem}).

\section{Main results and definitions}

\begin{defin}[$Y^x$-manifolds]\label{def:Yx}
Let $(M,g)$ be a Riemannian manifold and $x$ a point in $M.$ We say that $(M,g)$ is 
a {\it $Y^x$-manifold\/} if there exists $\overline \epsilon>0$ such that for every positive
$\epsilon<\overline \epsilon$ there exists $l=l(\epsilon)$ with $l>\epsilon$ such that for every 
unit speed geodesic $\gamma:\R\to M$ originating at $x$ at time $0$, 
the point $\gamma(l)$ is $\epsilon$ close to $x$. 
\end{defin}

In other words, in a $Y^x$-manifold for every sufficiently small neighborhood of $x$
there exists a moment of time $l$ when all the unit speed geodesics emitted from $x$ return 
back to this neighborhood (after first leaving the neighborhood).

\begin{defin}[$\widetilde Y^x$-manifolds]\label{def:tildeYx}
Let $(M,g)$ be a Riemannian manifold and $x$ a point in $M.$ We say that $(M,g)$ is a 
{\it $\widetilde Y^x$-manifold\/} 
if there exists $\overline \epsilon>0$ such that for every positive
$\epsilon<\overline \epsilon$ there exist $l=l(\epsilon)$  
with $l>\epsilon$ and $y=y(\epsilon)$ 
such that for every unit speed 
geodesic $\gamma:\R\to M$ originating at $y$ at time $0$, 
the point $\gamma(l)$ is $\epsilon$ close to $x$. 
For short we shall say that {\it all geodesics from $y$ in time $l$ focus within $\epsilon$ from $x.$\/}
\end{defin}

It immediately follows that every $Y^x_l$-manifold is a $Y^x$-manifold, 
while every $Y^x$-manifold is a $\widetilde Y^x$-manifold.

\begin{rem}[possible reformulations of Definitions~\ref{def:Yx} and~\ref{def:tildeYx}]
Our Corollary~\ref{anyepsilon} says that 
if the requirements described in Definitions~\ref{def:Yx} and~\ref{def:tildeYx}
are satisfied for all sufficiently small $\epsilon>0$, then 
in fact they are satisfied for all $\epsilon>0.$ Thus in both definitions one can forget the condition that 
the requirement is supposed to be satisfied only for sufficiently small $\epsilon$ and this would not 
change the class of manifolds described in the definition.

A much easier fact is that the requirement that the condition should be 
satisfied for all sufficiently small $\epsilon>0$, can be substituted by the condition that there exists a 
sequence $\{\epsilon_n\}_{n=1}^{\infty}$ of positive numbers with $\lim_{n\to \infty}\epsilon_n=0$ for which 
the condition is satisfied. 

Indeed if the condition is true for all positive $\epsilon<\overline \epsilon$, 
then it is also satisfied for all the members of
any such sequence $\{\epsilon_n\}_{n=1}^{\infty}$ with all $\epsilon_n<\overline \epsilon.$
On the other side, if $r$ is the radius of a geodesically convex normal neighborhood of $x$, then for every 
positive $\epsilon<\frac{r}{2}$ for which the condition is satisfied 
the corresponding $l(\epsilon)>\frac{r}{2}>\epsilon.$ Without the loss of generality 
we can assume that all the sequence members are less than $\frac{r}{2}.$
Then one chooses $\overline \epsilon=\epsilon_1.$ Now given any positive $\epsilon<\overline \epsilon$ choose 
$K\in \N$  so that $\epsilon_K<\epsilon$ and observe that if we put
$l(\epsilon)=l(\epsilon_K)$ (and $y(\epsilon)=y(\epsilon_K)$ if we talk about $\widetilde Y^x$-manifolds), then
the desired condition is satisfied. 
\end{rem}

\begin{remark} 
We do not know examples of $\widetilde Y^x$-manifolds that are not
$Y^x$-manifolds. 
Similarly we do not know examples of $Y^x$-manifolds that are not $Y^x_l$-manifolds for some nonzero $l.$
It may be that these three classes actually coincide, but an attempt to prove this runs into a problem we 
describe below. 

Given a Riemannian $\widetilde Y^x$-manifold $(M,g)$ let 
$\{\epsilon _n\}_{n=1}^{\infty}$ be a positive sequence with $\lim_{n\to\infty}\epsilon_n=0$. Then there is a sequence
$\{l_n\}_{n=1}^{\infty}$ of positive numbers and a sequence $\{y_n\}_{n=1}^{\infty}$ of points in $M$ 
such that for every geodesic $\gamma:\R\to M$ 
parametrized by arc length and originating at $y_n$, the point $\gamma(l_n)$ is $\epsilon_n$ close to $x$. 

Our Lemma~\ref{l:2} says that if $(\widetilde y, \widetilde l)$ is a limit point of the set 
$\{(y_n, l_n)\}_{n=1}^{\infty}\subset M\times \R$, then $\widetilde l>0$ and 
$(M,g)$ is a $Y^x_{2\widetilde l}$-manifold. 

Even though our Theorem~\ref{maintheorem} says that $M$ has to be compact, it is not clear if one can always 
choose $l_n$ so that they form a bounded sequence, and hence it is 
not clear whether the subset $\{(y_n, l_n)\}_{n=1}^{\infty}\subset M\times \R$ necessarily has a limit point.
This difficulty seems to be similar to~\cite[Question 7.70]{Besse} discussed in Remark~\ref{questions}.
\end{remark}

Our main result is a counterpart of the B\'erard-Bergery Theorem.

\begin{theorem}\label{maintheorem}
Let $M$ be a manifold of dimension at least $2$ such that there exists a complete Riemannian metric $g$ on $M$ and a point 
$x\in M$ with the property that $(M^m, g)$ is a $\widetilde Y^x$-manifold. 
Then $M$ is a closed manifold and $|\pi_1(M)|<\infty.$
\end{theorem}

Since every $Y^x_l$-manifold is a $\widetilde{Y}^x$-manifold, Theorem~\ref{maintheorem} implies 
the first two out of three properties of $Y^x_l$-manifolds in the B\'erard-Bergery Theorem. On the other hand, 
the proof of Theorem~\ref{maintheorem} contained in Section~\ref{proofmaintheorem} is completely different 
from that of the B\'erard-Bergery Theorem. It is based on Lorentzian geometry and notion of refocussing 
Lorentzian space-times introduced by Low~\cite{LowNullgeodesics, LowRefocussing}.

For reader's convenience in Section~\ref{LorentzGeometry} we review necessary notions of Lorentz 
geometry. In Section~\ref{quest} we 
discuss facts related to refocussing and open problems.

\section{Corollaries of Theorem~\ref{maintheorem} and other results}
\begin{fact}\label{cooltechnical}
Let $M^m, m=2,3$ be a closed manifold with finite fundamental group, then there is a complete Riemannian metric 
$g$ on $M$ and a point $x\in M$ such that $(M,g)$ is a $Y^x_{2\pi}$-manifold. 
\end{fact}
\begin{proof}
By the Thurston Elliptization Conjecture~\cite{Thurston} 
proved by Perelman~\cite{Perelman1, Perelman2}, a closed manifold $M$ of dimension $3$ with finite fundamental group is a quotient of the standard unit sphere $S^3$ by a
finite group of isometries. Thus $M$ inherits the quotient metric $g$ from the standard metric on the unit sphere. Clearly $(M, g)$ is a $Y^x_{2\pi}$-manifold. The proof of Fact~\ref{cooltechnical} in the case $m=2$ is similar to (but simpler than) the proof in the case $m=3$.  
\end{proof}

\begin{cor}\label{maincorollary}
Let $M^m, m=2,3$ be a manifold, such that there exists a complete Riemannian metric $g$ on $M$ and a point 
$x\in M$ with the property that $(M^m, g)$ is a $\widetilde Y^x$-manifold. 
Then the ring $H^*(M, \Q)$ is generated by one element. Moreover there exists a complete 
Riemannian metric $\wt g$ on $M$ such that $(M, \wt g)$ is a $Y^x_{2\pi}$-manifold.
\end{cor}

\begin{proof}
By Theorem~\ref{maintheorem} the manifold $M$ is closed and $|\pi_1(M)|<\infty.$  Hence Corollary~\ref{maincorollary} follows from 
Fact~\ref{cooltechnical} and Theorem~\ref{BB1}.
\end{proof}

If the condition in Definitions~\ref{def:Yx} and~\ref{def:tildeYx} is actually satisfied for all $\epsilon$, 
then it is also satisfied for all sufficiently small $\epsilon.$ The converse is also true, but requires some 
thinking, even though this condition is automatically satisfied for all $\epsilon> \diam(M,g)$.

\begin{cor}\label{anyepsilon}
If the condition in Definitions~\ref{def:Yx} and~\ref{def:tildeYx} of $Y^x$- and $\widetilde Y^x$-manifolds 
is satisfied for all sufficiently small $\epsilon>0$, then in fact it is satisfied for all $\epsilon>0$.
\end{cor}

\begin{proof}
We give the proof for the $\widetilde Y^x$-manifolds. The proof for $Y^x$-manifolds is similar and in fact 
simpler.

The case $\dim M=1$ is trivial.
Assume $(M^m,g), m>1$ is a $\widetilde Y^x$-manifold and let $\overline \epsilon>0$ 
be as in the definition of a $\widetilde Y^x$-manifold. Choose a sequence of positive numbers 
$\{\epsilon_n\}_{n=1}^{\infty}$, $\epsilon_n<\overline \epsilon_n$ with $\lim_{n\to \infty}\epsilon_n=0.$
Choose a sequence 
$\{l_n\}_{n=1}^{\infty}$ of positive numbers $l_n>\epsilon_n$ 
and a sequence $\{y_n\}_{n=1}^{\infty}$ of points in $M$ 
such that all geodesics from $y_n$ in time $l_n$ focus within $\epsilon_n$ from $x$. 

$M$ is compact by Theorem~\ref{maintheorem}. 
If the sequence 
$\{l_n\}_{n=1}^{\infty}$ is bounded, then $\{(y_n, l_n)\}_{n=1}^{\infty}$ contains a subsequence converging to 
some $(\widetilde y, \widetilde l)$. Lemma~\ref{l:2} says that $\widetilde l>0$ and $(M,g)$ is a 
$Y^x_{2\widetilde l}$-manifold. Then for a given $\epsilon>0$ one takes $k\in \N$ so that 
$2k\widetilde l>\epsilon$ and puts $y=x$ and $l(\epsilon)=2k\widetilde l$.
Thus the condition of Definition~\ref{def:tildeYx} is in fact satisfied for all 
$\epsilon >0$, rather than just for sufficiently small $\epsilon$.

If the sequence $\{l_n\}_{n=1}^{\infty}$ is not bounded, then we choose a monotonically increasing subsequence 
$\{l_{n_k}\}_{k=1}^{\infty}$ with $\lim_{k\to \infty}l_{n_k}=+\infty.$ Now take any 
$\epsilon>0$ that is not necessarily less than $\overline \epsilon.$ Choose $K$ such that $l_{n_K}>\epsilon$
and $\epsilon_{n_K}<\epsilon.$ Clearly the point $y=y_{n_K}$ and the positive number $l=l_{n_K}$ satisfy 
the requirements of 
definition~\ref{def:tildeYx} for the chosen $\epsilon>0.$
\end{proof}

\begin{cor}\label{setisclosed}
Let $(M,g)$ be a Riemannian manifold, then the possibly empty set 
$\widetilde Z=\{z \in M | (M,g)\text{ is a } \widetilde Y^z\}$ is a closed subset of $M.$
\end{cor}

We do not know if for a Riemannian manifold $(M,g)$, the possibly empty set 
$Z=\{z \in M | (M,g)\text{ is a } Y^z\}$ is always a closed subset of $M.$

\begin{proof}
If a connected $M$ has dimension one, then the statement is obvious, since $M$ is either diffeomorphic to 
$S^1$ or to $\R^1.$ In the first case $\widetilde Z=M$, in the second case $\widetilde Z=\emptyset.$
Similarly the statement is obvious if $\widetilde Z=\emptyset.$ So we consider the case 
$\widetilde Z\neq \emptyset$, and $\dim M>1.$

$M$ is compact by Theorem~\ref{maintheorem}. Thus there exists $L>0$ such 
that for every $p\in M$ the exponential map restricted to the radius $L$ ball centered at 
${\bf 0}\in T_pM$ is a diffeomorphism.
Take $p\in \widetilde Z$ and $\overline \epsilon(p)>0$ from the definition of a 
$\widetilde Y^p$-manifold. For each $\epsilon(p)<\overline \epsilon (p)$ we should have $y(p,\epsilon)\in M$ and 
$l(p,\epsilon)>\epsilon(p)$ such that if $\gamma(t)$ is a unit speed geodesic satisfying 
$\gamma(0)=y(p,\epsilon)$ then the distance $d_g(\gamma(l(p,\epsilon)), p)<\epsilon.$ 
By definition of $L$ we have that $l(p, \epsilon)>\frac{L}{2}.$ Thus we can put 
$\overline\epsilon (p)=\frac{L}{2}$ where the right hand side does not depend on $p\in \widetilde Z.$

It suffices to show that if $\widehat z$ is a limit point of $\widetilde Z$, then  
$\widehat z\in \widetilde Z.$ Put $\overline \epsilon (\widehat z)=\frac{L}{2}$
and take any $\epsilon< \overline \epsilon (\widehat z).$ Take $z_N\in \widetilde Z$ such that 
$d_g(\widehat z, z_N)<\frac{\epsilon}{2}.$
Since $z_N\in \widetilde Z$ and $\frac{\epsilon}{2}<\frac{L}{4}<\frac{L}{2}$ by the previous discussion 
there exist $y\in M$ and $l>\frac{L}{2}$ with the property that if $\gamma(t)$ is a unit speed geodesic 
satisfying $\gamma(0)=y$ then $d_g(\gamma(l), z_N)<\frac{\epsilon}{2}.$ Since 
$d_g(\widehat z, z_N)<\frac{\epsilon}{2}$ we have  $d_g(\gamma(l), \widehat z)<\epsilon$ for all 
unit speed geodesics $\gamma(t)$ satisfying $\gamma(0)=y.$ 
\end{proof}

\begin{cor}\label{coverings}
Let $x\in \widetilde M$ be a point in the connected total space of a Riemannian covering 
$\rho:(\widetilde M^m, \widetilde g)\to (M^m,g), m\geq 2$. Then the following two statements hold:
\begin{description}
\item[${\bf 1}$] $(\widetilde M, \widetilde g)$ is a $\widetilde Y^x$-manifold 
if and only if $(M,g)$ is a $\widetilde Y^{\rho(x)}$-manifold. 
\item[${\bf 2}$] $(\widetilde M, \widetilde g)$ is a $Y^x$-manifold 
if and only if $(M,g)$ is a $Y^{\rho(x)}$-manifold. 
\end{description}
\end{cor}

The statement of this Theorem is false when $M=S^1$ and $\widetilde M=\R^1.$
Note that we use Theorem~\ref{maintheorem} only in the proof of statement ${\bf 2}$.

\begin{proof} We prove statement ${\bf 1}$.
Assume that $(\widetilde M, \widetilde g)$ is a $\widetilde Y^x$-manifold. For a 
sufficiently small $\epsilon>0$ take $y\in M$ and $l>\epsilon$ such that all the geodesics from $y$ 
in time $l$ focus withing $\epsilon$ from $x.$
Since $\rho$ is a Riemannian covering, we get that all the geodesics from $\rho(y)$ in time $l$ focus within
$\epsilon$ from $\rho(x).$
Thus $(M,g)$ is a $\widetilde Y^{\rho(x)}$-manifold. 

Now we prove the other implication. Take $\overline \epsilon >0$ as in the definition of $(M,g)$ 
being a $\widetilde Y^{\rho(x)}$-manifold. We assume without the loss of generality that the exponential 
map $\exp_{\rho (x)}:T_{\rho(x)}M\to M$ restricted to the ball of radius $\overline  \epsilon$ centered at 
${\bf 0}\in T_{\rho (x)}M$ is a diffeomorphism.  We put $B$ to be the open neighborhood of $\rho(x)$ 
that is the image of the restriction of $\exp_{\rho (x)}$ map to this ball. 
Decreasing $\overline \epsilon$ if necessary we can and do assume that $B$ is trivially covered under $\rho$.

Choose any positive $\epsilon<\overline \epsilon$ and let $l>\epsilon$ and $y\in M$ be such that 
all geodesics from $y$ in time $l$ focus within $\epsilon$ from $x.$
Choose a unit speed geodesic $\gamma(t)$ such that 
$\gamma(0)=y.$ Put $z=\gamma(l)\in B.$ Put $\widetilde B$ to be the connected component of 
$\rho^{-1}(B)$ containing $x$ and put $\widetilde z$ to be the unique point of $\rho^{-1}(z)$ located within 
$\widetilde B.$ Let $\widetilde \gamma(t)$ be the lift of the path $\gamma(t)$ such 
that $\widetilde \gamma(l)=\widetilde z.$ 
Put $\widetilde y$ to be $\widetilde \gamma(0),$ so that $\rho(\widetilde y)=y.$
Since positive $\epsilon<\overline \epsilon$ was arbitrary, to finish the proof 
it suffices to show that all the geodesics from $\widetilde y$ in time $l$ 
focus within $\epsilon$ from $x.$
By our choice of $y$ and $l$  and since $\rho$ is a covering, the end point of each of these geodesic arcs 
of length $l$ starting from $\widetilde y$ is located within $\epsilon$ from one of the points in 
$\rho^{-1}(\rho(x))$ and thus within one of connected components of the preimage of $B$ under $\rho.$ 
These end points continuously depend on the initial directions of the geodesic arcs. 
Since $\dim M>1$,  the sphere of unit vectors in $T_{\widetilde y}\widetilde M$ is connected. 
Since $B$ is covered trivially under $\rho$, the end points of all these length $l$ geodesic arcs starting 
at $\widetilde y$ are located within epsilon 
from the same point in $\rho^{-1}(\rho(x))$ as the end point $\widetilde \gamma(l)$. This point is $x.$

Now we prove statement ${\bf 2}$. Clearly if $(\widetilde M, \widetilde g)$ is a $Y^x$-manifold, 
then $(M,g)$ is a $Y^{\rho(x)}$-manifold. 

Now we prove the other implication. Every  $(\widetilde M,\widetilde g)$ is the base space of the 
Riemannian cover by the total space of the universal Riemannian cover of $(M,g).$ Thus by the previously 
proved implication it suffices to prove the statement when $\rho$ is the universal Riemannian covering.

Let $r>0$ be such that $\exp_{\rho(x)}:T_{\rho(x)}M\to M$ restricted to the radius $r$ ball centered 
at ${\bf 0}\in T_{\rho(x)}M$ is a diffeomorphism. Let $\overline \epsilon>0$ 
be as in Definition~\ref{def:Yx}.
Choose any positive $\epsilon<\min(\frac{r}{2}, \overline \epsilon)$. 
Take $l_1$ such that all geodesics from 
$\rho(x)$ in time $l_1$ focus within $\frac{\epsilon}{2}$ from $\rho(x).$ Similar to the proof of 
Corollary~\ref{setisclosed}, we get that $l_1>\frac{r}{2}.$
As in the proof of the first statement of the Theorem we get that there is $x_1\in \rho^{-1}(\rho(x))$ 
such that 
all geodesics from $x_1$ in time $l_1$ focus within $\frac{\epsilon}{2}$ from $x.$ If $x_1=x$, then we found 
the desired $l=l_1.$ 

Assume that $x_1\neq x.$ 
Let $\Gamma$ be the group of deck transformations of the universal covering $\rho$. 
It acts transitively on $\rho^{-1}(\rho(x))$ 
and we put $\alpha_1\in \Gamma$ to be such that $x_1=\alpha_1(x).$
Since $\rho$ is a Riemannian covering, we get that for each $x'\in \rho^{-1}(\rho(x))$ all the geodesics 
from $\alpha_1(x')$ in time $l_1$ focus within $\frac{\epsilon}{2}$ from $x'.$

Since the geodesic flow $STM\times \R\to STM$ is continuous, we get that there is a small positive
$\widetilde \epsilon_1<\frac{\epsilon}{2}$, such that for all the points $y$ in the 
$\widetilde \epsilon_1$-ball centered at $x_1$ all the geodesics from $y$ in time $l_1$ focus within 
$\epsilon=2\frac{\epsilon}{2}$ from $x.$

Repeat the previous argument to find $l_2>0$ and $x_2\in \rho^{-1}(\rho(x))$ such that all geodesics from 
$x_2$ in time $l_2$ focus within $\widetilde \epsilon_1$ of $x_1.$ Thus all the geodesics from $x_2$ in time 
$l_2+l_1$ focus within $\epsilon$ of $x.$ Put $\alpha_2\in \Gamma$ to be such that $x_2=\alpha_2(x_1)$.
Then for every $x'\in \rho^{-1}(\rho(x))$ all the geodesics from 
$\alpha_2\alpha_1(x')$ in time 
$l_1+l_2$ focus within $\epsilon$ from $x';$ all the geodesics from $\alpha_2(x')$ in time $l_2$ 
focus within $\widetilde \epsilon_1<\epsilon$ of $\alpha_1(x')$; and all the geodesics 
from $\alpha_1(x')$ in time $l_1$ focus within $\frac{\epsilon}{2}<\epsilon$ of $x'$. Proceed by 
induction. 

By Theorem~\ref{maintheorem} $|\rho^{-1}(\rho(x))|<\infty.$ 
So at a certain step of the inductive process the newly chosen $x_j\in \rho^{-1}(\rho(x))$ will coincide 
with the previously chosen $x_i, i<j.$ Then all the geodesics from 
$x_i=x_j=\alpha_j \alpha_{j-1}\dots \alpha_{i+1}(x_i)$ in time $l_j+l_{j-1}+\dots +l_{i+1}$ focus within 
$\epsilon$ from $x_i$. Since $\rho$ is a Riemannian cover, we get that all the geodesics 
from $\alpha_j \alpha_{j-1}\dots \alpha_{i+1}(x)$ in time $l_j+l_{j-1}+\dots +l_{i+1}$ focus within 
$\epsilon$ from $x$.

\end{proof}

\section{Brief Introduction to Lorentzian Manifolds}\label{LorentzGeometry}

A \emph{Lorentzian manifold} $(X^{m+1}, g^L)$ is a Pseudo-Riemannian manifold whose metric tensor $g^L$ is of signature $(m, 1)$. In other words, each point $p\in X$ of a Lorentzian manifold $(X^{m+1}, g^L)$ has a coordinate neighborhood with coordinates $(x_1, ..., x_{m+1})$ such that the metric tensor $g^L|T_pX\times T_pX$ is of the form
\[
    dx_1^2 + dx_2^2 +\cdots + dx_{m}^2 - dx_{m+1}^2,
\]
where $T_pX$ is the tangent space of $X$ at $p$. 

A nonzero vector $v\in T_pX$ of a Lorentzian manifold $(X^{m+1}, g^L)$ is said to be {\em timelike,
non-spacelike, null \textup{(}lightlike\textup{)}, {\rm or} spacelike\/} if
$g^L(v,v)$ is negative, non-positive, zero, or positive,
respectively. A piecewise smooth curve $\gamma(t)$ is called {\em timelike,
non-spacelike, null, {\rm or} spacelike\/} if all of its velocity
vectors $\gamma'(t)$ are respectively timelike, non-spacelike, null, or
spacelike. 

For each $p\in X$ the set of all non-spacelike vectors in
$T_pX$ has two connected components that are hemicones. A
continuous \textup{(}with respect to $p\in X$\textup{)} choice of a hemicone of
non-spacelike vectors in $T_pX$ is called a {\em time
orientation\/} of $(X, g^L).$ The non-spacelike vectors from the chosen hemicones are called
{\em future pointing vectors.\/}
A piecewise smooth curve is said to be {\em future directed\/} if all of its
velocity vectors are future pointing. A connected time-oriented Lorentzian manifold without boundary is called a {\it space-time\/}.

For two events $x,y$ in a space-time $(X, g^L)$ we 
write $x\leq y$ if $x=y$ or if there is a piecewise smooth future
directed non-spacelike curve from $x$ to $y.$ For $x\in (X,g^L)$, the spaces
\[
J^+(x)=\{y\in X | x\leq y\} \qquad \mathrm{and}\qquad  J^-(x)=\{y\in X | y\leq x\}
\]
are called the {\em causal future \/} and {\em causal past of $x$\/} respectively. Two events $x,y$ are {\it causally related\/} if $y\in J^{\pm}(x).$ A space-time $(X^{m+1}, g^L)$ is {\it causal\/} if it does 
not have closed future directed non-spacelike curves.

An open set in a space-time is {\em causally convex\/} if there are no
future directed non-spacelike curves intersecting it in a disconnected set. A
space-time is {\em strongly causal\/} if every point in it has
arbitrarily small causally convex neighborhoods. A strongly causal
space-time $(X, g^L)$ is {\em globally hyperbolic\/} if $J^+(x)\cap
J^-(y)$ is compact for all $x,y\in X.$

A {\em Cauchy surface\/} $M$ is a subset of a space-time $(X, g^L)$ 
such that for every inextendible future directed non-spacelike curve $\gamma(t)$ in $X$ there
exists exactly one value $t_0$ of $t$ with $\gamma(t_0)\in M$.  A space-time can be shown
to be globally hyperbolic if and only if it admits a Cauchy surface,
see~\cite[pages 211-212]{HawkingEllis}.

\begin{usefulfacts}\label{usefulfacts}
$(a)$ Every Lorentzian manifold $(X,g^L)$ has a unique Levi-Civita
connection, see for example {\rm \cite[page 22]{BeemEhrlichEasley}.} This allows one to talk about geodesics
and about null-geodesics, i.e., geodesics whose velocity vector is null everywhere. An affine 
reparameterization of a null geodesic also is a null geodesic. However, contrary to the Riemannian geometry, 
null geodesics do not have a canonical parametrization. A curve is called a {\it pregeodesic\/} if it can be 
reparameterized to be a geodesic.

$(b)$ The pioneer result of Geroch~\cite{Geroch} says that every
globally hyperbolic space-time $(X^{m+1}, g)$ is {\em homeomorphic} to
$M^m\times \R$ where every $M\times t\subset X$ is a Cauchy surface. 

$(c)$ Bernal and Sanchez~\cite[Theorem 1]{BernalSanchez},~\cite[Theorem
1.1]{BernalSanchezMetricSplitting},~\cite[Theorem
1.2]{BernalSanchezFurther} have proved that a Cauchy surface $M$ in a globally hyperbolic 
space-time $(X^{m+1}, g^L)$ can always be chosen to be smooth and {\it spacelike\/} i.e., $g^L|TM$ is Riemannian. 
Moreover in this case $X$ is diffeomorphic to $M\times \R$, each slice $M\times t$ is a smooth spacelike Cauchy 
surface, and any two such Cauchy surfaces are diffeomorphic. They also proved~\cite{BernalSanchezCausal}
that in the definition of globally hyperbolic space-times it suffices to require that $(X^{m+1}, g^L)$ is 
causal rather than that it is strongly causal.

$(d)$ Let $(M, g)$ be a Riemannian manifold, 
and let $f:(\alpha, \beta) \to (0, +\infty)$ be a smooth positive function, where 
$-\infty\leq \alpha<\beta\leq +\infty$. Then 
the warped product space-time $(M\times (\alpha, \beta), f(t)g\oplus -dt^2)$ is globally hyperbolic and each 
$M\times t$ is a smooth spacelike Cauchy surface, see ~\cite[Theorem 3.66]{BeemEhrlichEasley}.

$(e)$ Two Lorentz manifolds $(X_1, g^L_1)$ and $(X_2, g^L_2)$ are said to be {\it conformal equivalent\/}
if there exists a diffeomorphism $f:X_1\to X_2$ and a positive smooth function $\Omega:X_1\to (0, +\infty)$
such that $g_1^L=\Omega f^*(g^L_2).$ If $\gamma$ is a timelike or spacelike or null curve in $(X_1, g^L_1)$,
then clearly $f(\gamma)$ is respectively a timelike or spacelike or null curve in $(X_2, g^L_2)$. Moreover
if $\gamma$ is a null pregeodesic, then $f(\gamma)$ also is a null 
pregeodesic~\cite[Lemma 9.17]{BeemEhrlichEasley}. The similar statement is generally false for 
spacelike and timelike pregeodesics.

\end{usefulfacts}

\section{Refocussing and examples}\label{refocussing}
\begin{defin}	[Strongly refocussing Lorentzian manifolds]
We say that a Lorent\-zian manifold $(X^{m+1}, g^L)$ is {\it strongly refocussing at $x\in X$\/} 
if there exists $y\in X$ such that for every (inextendible) null geodesic 
$\nu(t)$ with $\nu(0)=y$ there exists nonzero $\tau=\tau(\nu)$ such that $\nu(\tau)=x.$ 
Note that this $\tau$ may and generally does depend on the choice of the null geodesic $\nu.$ 

We say that a Lorentzian manifold is {\it strongly refocussing} if it is strongly refocussing at some point. 
\end{defin}

We require $\tau\neq 0$ since otherwise we always have refocussing via choosing $y=x$ and $\tau=0.$ 
This definition means that all the light rays through $y$ also pass through $x$ (for nontrivial reasons).

\begin{defin}[Weakly refocussing Lorentzian manifolds] 
We say that $(X^{m+1}, g^L)$ is {\it (weakly) refocussing at $x\in X$\/} if there exists 
open $U\ni x$ such that given any open $V$ with $x\in V\subset U$ there exists $y\not \in V$ 
such that all the null geodesics through $y$ pass through $V.$  Note that these null geodesics are not
required to pass through $x$.

We say that a Lorentzian manifold is {\it weakly refocussing} if it is weakly refocussing at some point. 
\end{defin}

This definition was 
introduced by Low~\cite{LowNullgeodesics, LowRefocussing} for the physically 
interesting strongly causal space-times.

\begin{rem}\label{conformal}
Let $(X_1, g^L_1)$ and $(X_2, g^L_2)$ be conformal space-times. Let $f:X_1\to X_2$ be a diffeomorphism and 
$\Omega:X_1\to (0,+\infty)$ be a smooth positive function such that $\Omega f^*(g^L_2)=g^L_1$. If $\gamma(t)$ 
is a null pregeodesic for $(X_1, g^L_1)$, then $f(\gamma(t))$ is a null pregeodesic for $(X_2, g^L_2)$, 
see~\cite[Lemma 9.17]{BeemEhrlichEasley}. 

Thus if $(X_1, g^L_1)$ is refocussing 
(respectively strongly refocussing) at 
$x\in X_1$, then $(X_2, g^L_2)$ is refocussing (respectively strongly refocussing) at $f(x).$
In particular if  $(X_1, g^L_1)$ and $(X_2, g^L_2)$ are conformal equivalent, then one is refocussing
exactly when the other one refocussing, and the same is true for strong refocussing.
\end{rem}

\begin{example}[Chernov-Rudyak construction~\cite{ChernovRudyak} of strongly refocussing space-times]\label{examplerefocussing} 
Let $(M,g)$ be a $Y^x_l$ manifold for some $x\in M$ and nonzero $l\in \R.$ Consider the Lorentzian product manifold 
$(X^{m+1}, g^L)=(M\times \R, g\oplus -dt^2)$. Then all the null geodesics 
through $(x, t-l)$ pass through $(x,t)$. Thus the globally hyperbolic space-time $(X^{m+1}, g^L)$ is strongly refocussing at $(x, t)$ for each $t\in \R$ (see \cite[Section 11, Remark 7]{ChernovRudyak}). 
\end{example}

Example~\ref{examplerefocussing} can be modified to yield a strongly refocussing Lorentzian manifold with a metric that is not a product metric. Indeed, let $U$ be an open neighborhood of the singular hypersurface in $X^{m+1}$ covered by the union of the arcs of the null geodesics from $(x, -l)$ to $(x,0).$ Let $g^L_U$ be any Lorentzian metric that equals to $g\oplus-dt^2$ on $U$. Then $(M\times \R, g^L_U)$ is strongly refocussing at $(x,0)$. This gives a vast collection of strongly 
refocussing Lorentzian manifolds with a metric that is not the product metric.

\begin{example}[Weakly refocussing space-times]
From the proof of Theorem~\ref{maintheorem} it is easy to see that if $(M,g)$ is a $\widetilde Y^x$-manifold for some $x\in M$, then $(M\times \R, g\oplus -dt^2)$ is refocussing 
at $(x, t)$ for each $t\in \R.$ 
\end{example}

\begin{example}[Kinlaw~\cite{Kinlaw} 
example of globally hyperbolic space-times that are refocussing but not strongly 
refocussing at a point]\label{Kinlaw} 
Let $g$ be the standard metric on a unit sphere $S^m\subset \R^{m+1}.$ Then $(X_1, g^L_1)=(S^m\times (-\pi, \pi), g\oplus -dt^2)$
contains a codimension one submanifold $\Sigma=\{(x, 0)|x\in S^m\}$, such that $(X_1, g^L_1)$ is refocussing 
but not 
strongly refocussing at each point of $\Sigma.$  Note that $(X_1, g^L_1)$ is strongly refocussing 
at $(x, -\delta)$ for 
small $\delta>0$ and thus the globally hyperbolic manifold $(X_1, g^L_1)$ is strongly refocussing. 

It is easy to see that this example is neither spacelike, nor timelike, nor null geodesically complete.
However by~\cite[Lemma 9.17]{BeemEhrlichEasley} it is conformal equivalent to a globally  hyperbolic 
space-time $(X_2, g^L_2)$ 
that is null and timelike geodesically complete. By Remark~\ref{conformal} this $(X_2, g_2^L)$ also has a 
hypersurface formed by points such that $(X_2, g^L_2)$ is refocussing 
but not  strongly refocussing at these points. 
\end{example}

The following Theorem is close in spirit to our Corollary~\ref{setisclosed}.

\begin{theorem}[Kinlaw~\cite{Kinlaw}]
Let $(X^{m+1}, g^L)$ be a strongly causal space-time. Then the possibly 
empty set $\widetilde Z=\{z\in X|(X, g^L)\text{ is refocussing at } z\}$ is a closed subset of $X.$
\end{theorem}

Example~\ref{Kinlaw} shows that the set $Z=\{z\in X|(X, g^L)\text{ is strongly refocussing at } z\}$
does not have to be a closed subset of a strongly causal $(X^{m+1}, g^L).$

\section{Proof of Theorem~\ref{maintheorem}}\label{proofmaintheorem}
Let $(M, g)$ be a $\widetilde Y^x$-manifold for some point $x\in M$. Let $\pr:STM\to M$ denote the tangent unit sphere bundle of $M$. The fiber of $\pr$ over a point $y\in M$ is denoted by $ST_yM.$ For $v\in STM$, let $\gamma_v:\R\to STM$ be the unique 
unit speed geodesic with $\dot \gamma_v(0)=v.$ There are smooth maps 
\[
p\colon STM\times \R\to M,
\] 
\[
p\colon v\times \tau\mapsto \gamma_v(\tau)
\]
and 
\[
q\colon STM\times \R\to STM,
\]
\[
q\colon v\times \tau\mapsto \dot\gamma_v(\tau).
\]
 
We recall that there is a positive real number $\overline \epsilon$ such that the property in the definition of $\widetilde Y^x$-manifolds holds for all $\epsilon$ with $0<\epsilon<\overline \epsilon.$ Let $\{\epsilon_n\}_{n=1}^{\infty}$ be a sequence of positive numbers 
$\epsilon_n<\overline \epsilon$ with $\lim \epsilon_n =0.$
Since $(M,g)$ is a $\widetilde Y^x$-manifold, there exist sequences of positive real numbers 
$\{l_n\}_{n=1}^{\infty}$, $l_n>\epsilon_n$ and points $\{y_n\}_{n=1}^{\infty}$ such that 
\[
\Im (p\,|_{ST_{y_n}M\times l_n})\subset B(x, \epsilon_n),
\]
where $B(x, \epsilon_n)$ denotes the open ball in $M$ about $x$ of radius $\epsilon_n$. 

To begin with let us assume that the sequence $\{(y_n, l_n)\}_{n=1}^{\infty}$ of points in $M\times \R$ 
has no convergent subsequence. Then, as we show in Lemma~\ref{l:1}, the globally hyperbolic Lorentzian 
product manifold $(X, g^L)=(M\times \R, g\oplus -dt^2)$ 
is refocussing. By a result of Low, its Cauchy surface $M\times 0=M$ is a closed manifold, 
see~\cite{LowNullgeodesics, LowRefocussing}, 
\cite[Section 11, Proposition 6]{ChernovRudyak}. On the other hand, the result of Rudyak and the first author says that the 
fundamental group of the Cauchy surface $M\times 0$ has to be finite, see~\cite[Theorem 15]{ChernovRudyak}. This completes the proof of Theorem~\ref{maintheorem} under the assumption that the sequence $\{(y_n, l_n)\}$ has no convergent subsequence. 

Suppose now that the sequence $\{(y_n, l_n)\}$ has a convergent subsequence. Then, by Lemma~\ref{l:2} below, the manifold $(M, g)$ is a $Y^x_l$-manifold for some $l$. Hence, in this case, the conclusion of Theorem~\ref{maintheorem} immediately follows from Lemma~\ref{l:1} and the B\'erard-Bergery Theorem~\ref{BB1}. 

In the rest of the section we prove Lemmas~\ref{l:1} and \ref{l:2}.

\begin{lem}\label{l:1}
Suppose that the sequence $\{(y_n, l_n)\}$ does not have a convergent subsequence. Then the globally hyperbolic Lorentzian product manifold $(X, g^L)=(M\times \R, g\oplus -dt^2)$ 
is refocussing at $(x, 0)$. 
\end{lem}
\begin{proof}
Put $g^R=g\oplus dt^2$ to be the product Riemannian metric on 
$M\times \R$ and put 
\[
U=B\bigl( (x, 0), \overline \epsilon\bigr)=\{(y, t)\in M\times \R\,|\, 
d_{g^R}\bigl ((y,t), (x,0)\bigr)<\overline \epsilon \}
\] 
to be the open ball neighborhood of $(x,0)\in M\times \R$ of radius $\overline\epsilon.$  Let $V\subset U$ be any neighborhood of $(x, 0)$. Put 
\[
\widetilde V=\{y\in M| y\times 0\in V\}
\] 
to be the open neighborhood of $x$ and put $\epsilon>0$ with $\overline \epsilon> \epsilon$ to be such 
that 
\[
B(x, \epsilon)=\{y|d_g(x,y)<\epsilon\}\subset \widetilde V.
\]

Since the sequence $\{(y_n, l_n) \}$ does not have a convergent subsequence 
and 
\[
\lim_{n\to \infty}\epsilon_n=0,
\] 
there exists 
a positive integer $N$ such that $\epsilon_N<\epsilon $ and $(y_N, l_N)\not \in U.$ 
If $w\in T_{(y, \tau)}(M\times \R)$ is a null vector with components 

\[
(w_M, w_{\R})\in T_yM\oplus T_{\tau}\R=T_{(y, \tau)}(M\times \R),
\] 
then $g(w_M, w_M)=-w_{\R}\cdot w_{\R}$, 
where $\cdot$ is the standard Riemannian metric on $\R^1.$ Since $g\oplus -dt^2$ is a Lorentzian product metric, the geodesics in $(M\times \R, g\oplus -dt^2)$ should project 
to geodesics in $(M,g)$ and in $(\R, -dt^2)$. 
Thus all the null geodesics through $(y_{N}, -l_{N})\not \in V$ intersect $\wt V\times 0\subset V$ 
and hence they all intersect $V.$ Hence $(M\times\R, g\oplus-dt^2)$ is refocussing.

\end{proof}

\begin{lem}\label{l:2}
Suppose that the sequence of points 
\[
\{(y_n, l_n)\}_{n=1}^{\infty}\subset M\times \R
\] 
contains a subsequence $\{(y_{n_k}, l_{n_k})\}_{k=1}^{\infty}$ converging to a point $(\wt y, \wt l).$ Then $\widetilde l\neq 0$, and $(M,g)$ is a $Y^x_{2\wt l}$-manifold.
\end{lem}
\begin{proof}
Put $r>0$ to be such that the exponential map $\exp_x:T_xM\to M$ restricted to the radius $r$ ball centered 
at ${\bf 0}\in T_xM$ is a diffeomorphism. Then each $l_n>r$ and thus the limit value $\widetilde l$ 
is non-zero. 
 
Without loss of generality we can assume that each point $y_{n_k}$ belongs to a prescribed geodesically convex neighborhood $W$ of $\wt y.$ 
For $v\in ST_{\wt y}M$ let $v_{n_k}\in ST_{y_{n_k}}M$ denote the vector obtained by the parallel transport of $v$ along the unique geodesic in $W$ connecting $\wt y$ to $y_{n_k}.$ Then 
\[
\lim_{k\to \infty}(v_{n_k}, l_{n_k})=(v, \widetilde l).
\]
Since the map $p$ is continuous, its values $p(v_{n_k}, l_{n_k})$ converge to $p(v, \wt l)$ as $k\to \infty$. On the other hand, each point $p(v_{n_k}, l_{n_k})$ is $\epsilon_{n_k}$-close to the point $x$. In view of the convergence $\epsilon_{n_k}\to 0$, we conclude that $p(v, \widetilde l)=x$. Consequently,
\[
\Im p\,|_{ST_{\wt y}M\times \wt l}=x,
\]
i.e., $\gamma(\widetilde l)=x$ for each geodesic $\gamma$ emitted from $\wt y$ with $|\dot\gamma(0)|=1$. Thus 
\[
q|_{ST_{\wt y}M\times \wt l}:ST_{\wt y}M=S^{m-1}\to ST_{x}M=S^{m-1}
\] 
is a smooth embedding and hence a diffeomorphism for dimensional reasons. Consequently, 
\[
\Im p\,|_{ST_{x}M\times \wt l}=\wt y.
\]  
This implies that $(M,g)$ is a $Y^x_{2\wt l}$-manifold. 
\end{proof}

\begin{remark} As it has been explained in Example~\ref{examplerefocussing}, 
one deduces that the globally hyperbolic Lorentzian product manifold $(X, g^L)=(M\times \R, g\oplus -dt^2)$ is refocussing at $(x, 0)$ (and hence at each $(x, t)$ with $t\in \R$ 
for the reason of symmetry).
\end{remark}

\begin{remark} Theorem~\ref{maintheorem} can be also proved so that its proof is independent of the B\'erard-Bergery Theorem, which we used in the proof of Theorem~\ref{maintheorem} under the hypothesis of Lemma~\ref{l:2}. Indeed, suppose that the sequence of points $\{(y_n, l_n)\}$ contains a sub-sequence converging to a point $(\widetilde y, \widetilde l)$. Since the limit value $\widetilde l$ is different from zero, we may still apply the argument of Lemma~\ref{l:1} to complete the proof of Theorem~\ref{maintheorem}. On the other hand, the statement of Lemma~\ref{l:2} is somewhat stronger than what we can deduce using the argument in Lemma~\ref{l:1} as it asserts that all geodesics emitted from $\widetilde y$ return precisely to the point $\widetilde y$ at the moment $2\widetilde l$.   
\end{remark}

\section{Intriguing facts related to $Y^x$- and $\widetilde Y^x$-Riemannian manifolds, 
refocussing space-times, positive Legendrian isotopy and open questions}\label{quest}
\subsection{$Y^x_l$-Riemannian manifolds, 
causality in space-times, Low Conjecture and the Legendrian Low Conjecture}\label{LowConjecture}
Low Conjecture~\cite{Low0}, \cite{Low1}, \cite{Low3},~\cite{LowLegendrian} and the 
Legendrian Low conjecture due to Natario and Tod~\cite{NatarioTod} reformulate causality in a globally 
hyperbolic space-time $(X^{m+1},g^L)$ in terms of link theory.  
Basically they ask if it is true that when the Cauchy surface is diffeomorphic to an open subset of 
$\R^2$ or of $\R^3$, 
then two events $x,y\in X$ are causally 
related if and only if the spheres of null geodesics passing 
through $x$ and $y$ are linked (in the appropriate sense) in the contact manifold  of all 
non-parameterized future pointing null geodesics in $(X^{m+1},g)$. 
This motivated a problem  communicated by Penrose on Arnold's problem lists~\cite[Problem~8]{ArnoldProblem},
\cite[Problem 1998-21]{ArnoldProblemBook}.

Stefan Nemirovski and the first author~\cite[Theorem A, Theorem C]{ChernovNemirovski} proved the 
Low and the Legendrian Low conjectures. They also 
proved~\cite[Theorem 10.4]{ChernovNemirovskiSecondPaper} 
that the statements of these conjectures remain true for all 
globally hyperbolic space-times $(X^{m+1}, g^L), m>1$ such 
that the total space of the universal cover of its Cauchy 
surface $M^m$ is an open manifold. 

If $(M,g)$ is a $Y^x_l$ Riemannian manifold, then 
these conjectures are false in the strongly refocussing $(M\times \R, g\oplus -dt^2)$, 
see~\cite[Example 10.5]{ChernovNemirovskiSecondPaper}.
In the physically most interesting case of  a 
$(3+1)$-dimensional globally hyperbolic space-time $(X^{3+1}, g^L)$ 
we get that if the Legendrian Low conjecture fails for $(X^{3+1}, g^L)$, then 
the Cauchy surface of $X$ admits a Riemannian metric making it into a $Y^x_l$-manifold, 
see~\cite[page 1322]{ChernovNemirovski}.

\subsection{Topology of a refocussing globally hyperbolic space-time}\label{topology}
An interesting question is what should be the topology of a Cauchy surface $M$ of a refocussing 
globally hyperbolic $(X^{m+1}, g^L)$. Low~\cite{LowNullgeodesics, LowRefocussing} 
proved that $M$ has to be a closed 
manifold, see also~\cite[Section 11, Proposition 6]{ChernovRudyak}. Rudyak and the first author 
proved that 
the universal Lorentzian cover of a refocussing globally hyperbolic space-time 
$(X^{m+1}, g^L), m>1$ is a refocussing 
globally hyperbolic space-time, see~\cite[Theorem 14]{ChernovRudyak}. Thus $|\pi_1(M)|<\infty.$

It is interesting to know if the third implication 
of the B\'erard-Bergery Theorem holds for a Cauchy surface $M$ of a globally hyperbolic refocussing space-time
$(X^{m+1}, g), m>1$, i.e.,~is it true that the ring $H^*(M^m , \Q)$ is generated by one element? This is true 
for $\dim M=2,3$. Indeed Fact~\ref{cooltechnical} says that such $M$ admits a Riemannian metric $g_q$ making 
$(M, g_q)$ into a $Y^x_{2\pi}$-manifold. Now B\'erard-Bergery Theorem~\cite{BerardBergery}, 
\cite[Theorem 7.37, page 192]{Besse} says that the ring $H^*(M, \Q)$
is generated by one element.

Paul Kinlaw~\cite{Kinlaw} 
noticed that this discussion implies the following intriguing observation: 
if a globally hyperbolic 
space-time $(X^{m+1}, g^L), m=2,3$ is refocussing, 
then it admits a globally hyperbolic Lorentzian metric $\widetilde g^L=g_q\oplus -dt^2$ such 
that $(X^{m+1}, \widetilde g^L)$ is strongly refocussing.

We do not know examples of globally hyperbolic space-times that are refocussing but 
not strongly refocussing. However Example~\ref{Kinlaw} shows that the situation is quite 
nontrivial.

\subsection{$\widetilde Y^x$- and $Y^x$-manifolds and positive Legendrian isotopy} 
Since $(M,g)$ is a Riemannian manifold we have the natural identification $STM=ST^*M.$ The 
spherical cotangent bundle $ST^*M$ has a natural contact structure and the $S^{m-1}$-fiber $S_x$ of 
$ST^*M\to M$ over a point $x\in M$ is a Legendrian submanifold. 
For each $t$ the map $q|_{ST^*M\times t}:ST^*M\to ST^*M$ preserves the contact structure and hence it maps 
Legendrian submanifolds to Legendrian submanifolds. Moreover the map $\phi:S_x\times[0,\infty)\to ST^*M$
defined by $\phi(z,t)=q(z,t)$ is a positive Legendrian isotopy, i.e. it is a Legendrian isotopy such that 
the evaluation of the contact form on the velocity vectors of the trajectory curves 
$\phi_z(t)=\phi(z,t):[0,\infty)\to ST^*M$, $z\in S_x$ is positive. 

If $(M,g)$ is a $Y^x_l$-manifold, then $\phi:S_x\times [0,l]\to ST^*M$ is a positive Legendrian isotopy of 
the fiber $S_x$ to itself. If a Cauchy surface $M^m, m>1$ of a globally hyperbolic space-time 
$(X^{m+1}, g)$ is such that there is no positive Legendrian isotopy of an $S^{m-1}$-fiber of $ST^*M$ to 
itself, then the Legendrian Low conjecture holds for $(X^{m+1}, g)$, 
see~\cite[Section 7]{ChernovNemirovski} and \cite[proof of Theorem 10.4]{ChernovNemirovskiSecondPaper}.
In~\cite[Corollary 8.1]{ChernovNemirovskiSecondPaper} 
Nemirovski and the first author 
proved that if $ST^*M$ admits a positive Legendrian isotopy of $S_x$ to itself, then $M$ is compact
and has finite $\pi_1(M)$, i.e.~the universal cover of $M$ is compact.  
In particular this gives yet another proof of the first two statements of 
B\'erard-Bergery Theorem~\ref{BB1}. A question in~\cite[Example 8.3]{ChernovNemirovskiSecondPaper} asks 
whether the existence of a positive Legendrian isotopy of $S_x$ to $S_x$
implies that the rational cohomology ring $H^*(M, \Q)$ is generated by one element.

It may be that the result of~\cite{ChernovNemirovskiSecondPaper} can be 
strengthened to show that if the universal cover of $M$ is not compact, then 
given two not necessarily distinct points $x,y\in M$ and a sufficiently small neighborhood 
$U$ of $S_x$ there is no positive Legendrian isotopy $\phi:S_y\times [0,1]\to ST^*M$ such that 
$\Im(\phi(S_y\times \frac{1}{2}))\cap U=\emptyset$ and $\Im(\phi(S_y\times 1))\subset U.$
If such a result holds it would give another proof of our Theorem~\ref{maintheorem}. 
One can also ask the question whether 
the existence of such a positive Legendrian isotopy implies that the ring $H^*(M,\Q)$ is generated by one element.

If $(M,g), m=2,3$ is a $\widetilde Y^x$- or a $Y^x$-manifold then 
the ring $H^*(M,\Q)$ is generated by one element, 
see~Corollary~\ref{maincorollary}. So one can ask if it holds 
in all dimensions. This question does not seem to be immediately reducible to the 
question in~\cite[Example 8.3]{ChernovNemirovskiSecondPaper}. 
 
Indeed given a $Y^x$-manifold $(M,g)$ and a sequence $\{\epsilon_n\}_{n=1}^{\infty}$ 
of sufficiently small positive numbers converging to zero, 
put $\{l_n\}_{n=1}^{\infty},$ $l_n>\epsilon_n$ to be a sequence 
as in Definition~\ref{def:Yx}. Put $B(x,\epsilon_n)$ be the metric 
ball of radius $\epsilon_n$-centered at $x.$ Clearly $\Im p|_{S_x\times l_n}$ is a Legendrian submanifold of 
$ST^*B(x,\epsilon_n)\subset ST^*M$ that can be obtained from $S_x$ by a positive Legendrian isotopy 
within $ST^*M.$
However it is not clear if this 
submanifold can be deformed to $S_x$ by a positive Legendrian isotopy inside $ST^*B(x,\epsilon_n)$ 
or even inside $ST^*M$. A similar difficulty arises for $\widetilde Y^x$-manifolds.

{\bf Acknowledgments.} 
The first author is very thankful to Stefan Nemirovski and Robert Low 
for the very useful mathematical discussions.

\end{document}